\documentclass[12pt,titlepage]{article}
\usepackage{graphicx,color}
\usepackage{amsmath,amssymb,amsthm}
\input epsf

\newtheorem{thm}{Theorem}

\newtheorem{cor}{Corollary}

\newtheorem{prop}{Proposition}

\newtheorem{defi}{Definition}

\def\beq{\begin{equation}}\def\eeq{\end{equation}}
\def\beqn{\begin{eqnarray}}\def\eeqn{\end{eqnarray}}

\def\qed{\ifhmode\unskip\nobreak\fi\quad\ifmmode\Box\else$\Box$\fi}

\newcommand\mbf[1]{\mbox{\boldmath$#1$}}

\title{Shannon capacity and the categorical product}

\author{{\bf G\'abor Simonyi}$^{a,b,}$\thanks{Research partially supported by the
    National Research, Development and Innovation Office (NKFIH)
    grants K--116769 and K--120706.
}
\\ \\
$^a$Alfr\'ed R\'enyi Institute of Mathematics\\ \\
$^b$Department of Computer Science and Information Theory,\\ Budapest University of Technology and Economics\\ \\
 {\tt simonyi@renyi.hu}}
\date{}

\begin{document}
\maketitle

\begin{abstract}
Shannon OR-capacity $C_{\rm OR}(G)$ of a graph $G$, that is the traditionally
more often used Shannon AND-capacity of the complementary graph, is a
homomorphism monotone graph parameter satisfying $C_{\rm OR}(F\times
G)\le\min\{C_{\rm OR}(F),C_{\rm OR}(G)\}$ for every pair of graphs, where $F\times G$ is the categorical product of graphs $F$ and $G$. Here we
initiate the study of the question when could we expect equality in this
inequality. Using a strong recent result of Zuiddam, we show that if this
"Hedetniemi-type" equality is not satisfied for some pair of graphs then
the analogous equality is also not satisfied for this graph pair by some
other graph invariant that has a much ``nicer'' behavior concerning some
different graph operations. In particular,
unlike Shannon capacity or the chromatic number, this other invariant is both
multiplicative under
the OR-product and additive under the join operation, while it is also
nondecreasing along graph homomorphisms.
We also present a natural lower bound on $C_{\rm OR}(F\times G)$ and elaborate
on the question of how to find graph pairs for which it is known to be
strictly less, than the upper bound $\min\{C_{\rm OR}(F),C_{\rm OR}(G)\}$. We
present such graph pairs using the properties of Paley graphs.

\end{abstract}

\section{Introduction}
\message{Introduction}

For two graphs $F$ and $G$, their categorical (also called tensor or weak direct) product $F\times G$ is defined by $$V(F\times G)=V(F)\times V(G),$$
$$E(F\times G)=\{\{(x,u),(y,v)\}: x,y\in V(F), u,v\in V(G), \{x,y\}\in E(F), \{u,v\}\in E(G)\}.$$

Hedetniemi's more than half a century old conjecture that was refuted
recently by Yaroslav Shitov \cite{YS} (cf. also \cite{TZ} and \cite{HW} for
further developments) stated that the chromatic
number of $F\times G$ would be equal to the smaller of the chromatic numbers of $F$ and
$G$, i.e., that $$\chi(F\times G)=\min\{\chi(F),\chi(G)\}.$$
It is easy to see that the right hand side above is an upper bound on the left
hand side. If $c:V(F)\to \{1,2,\dots,\chi(F)\}$
is a proper coloring of $F$ then $c': (x,u)\mapsto c(x)$ is a proper
coloring of $F\times G$ proving $\chi(F\times G)\le \chi(F)$. As the same
argument works if we start with a proper coloring of $G$, this proves the
claimed inequality. Thus the real content of the conjecture was that the right
hand side is also a lower bound on $\chi(F\times G)$. Though not true in
general, this holds in several special cases. In particular, it is easy to
prove when $\min\{\chi(F),\chi(G)\}\le 3$ and it is also known to hold when
this value is $4$. The latter, however, is a highly nontrivial result of El
Zahar and Sauer
\cite{EZS} and the general case was wide open until the already mentioned
recent breakthrough by Shitov \cite{YS}. For several related results, see the survey papers \cite{Sauersur, Tardsur, Zhusur}.

\medskip\par\noindent
A map $f:V(T)\to V(H)$ between the vertex sets of graphs $T$ and $H$ is called
a graph
homomorphism if it preserves edges, that is, if $f$ satisfies
$\{a,b\}\in E(T)\Rightarrow \{f(a),f(b)\}\in E(H)$. The existence of a graph
homomorphism from $T$ to $H$ is denoted by $T\to H$.

\medskip\par\noindent
Behind the validity of the inequality
$\chi(F\times G)\le\min\{\chi(F),\chi(G)\}$ is the fact that both
$F\times G\to F$ and $F\times G\to G$ holds (just take the projection maps),
while it is generally true, that $T\to H$ implies $\chi(T)\le \chi(H)$.

\medskip\par\noindent
This suggests that if $p(G)$ is any graph parameter which is monotone nondecreasing under
graph homomorphism, that is for which $T\to H$ implies $p(T)\le p(H)$,
then an analogous question to Hedetniemi's conjecture is meaningful for it:
we automatically have $p(F\times G)\le\min\{p(F),p(G)\}$ and one may ask
whether equality holds. (If it does, we will say that $p$ satisfies the Hedetniemi-type equality.)

\medskip\par\noindent
Theorems of this kind already exist. A famous example is Zhu's celebrated result known as the
fractional version of Hedetniemi's conjecture \cite{Zhufrac}. It shows
equality above in case $p(G)$ is the fractional chromatic number. More
recently Godsil, Roberson, \v{S}amal, and Severini \cite{GRSS} proved a
similar result for the Lov\'asz theta number of the complementary graph and
also investigated the analogous question for a closely related parameter
called vector chromatic number introduced in \cite{KMS}.
(The vector chromatic number is different from the {\em strict vector chromatic
number} which is known to be identical to the Lov\'asz theta number of the
complementary graph, cf. \cite{GRSS}, \cite{LLGG}.)
They conjectured that the Hedetniemi-type
equality also holds for this parameter and proved it in special cases. In a
follow-up paper by Godsil, Roberson, Roomey, \v{S}amal, and Varvitsiotis
\cite{GRRSV} the latter conjecture is proved in general as well.

Both the fractional chromatic number and the Lov\'asz theta number of the
complementary graph are well-known upper bounds on the Shannon OR-capacity of
the graph which is the usual Shannon capacity, or Shannon AND-capacity of the
complementary graph. This is also a homomorphism-monotone parameter, so the
Hedetniemi-type question is meaningful for it. In this paper we initiate the
study of this question. Given the very complex behavior of Shannon capacity
there seems to be little reason to believe that the Shannon OR-capacity would
also satisfy the analogous equality. However, if one has to argue why it
looks unlikely, then the first argument that comes to mind is that
Shannon OR-capacity famously does not satisfy two other simple equalities that
both the fractional chromatic number and the Lov\'asz theta number of the
complementary graph do. (For more details of this, see the next section.)
Our main observation is that this argument is rather weak. This will be a
consequence of a strong recent result by Jeroen Zuiddam \cite{Zuiddam}.

We will also elaborate on the question of finding graph pairs that provide
potential counterexamples to
the mentioned equality. This turns out to be challenging as well, largely
because of the lack of knowledge about the general behaviour of Shannon
capacity. We will give a natural lower bound on $C_{\rm OR}(F\times G)$ in Section~\ref{sect:lbitc}.
If we want to "test" whether Shannon OR-capacity satisfies the Hedetniemi-type
equality in nontrivial cases, we need some graph pairs $(F,G)$, for which our
lower bound is strictly smaller than the upper bound
$\min\{C_{\rm OR}(F), C_{\rm OR}(G)\}$.
Since the Shannon capacity value is not known in too many nontrivial cases, finding such graph pairs is not entirely trivial. We will present some graph pairs with this property in the second subsection of Section~\ref{sect:lbitc}.

\section{Shannon OR-capacity} \label{sect:shanor}

The Shannon capacity of a graph involves a graph product which is different of
the categorical product that appears in Hedetniemi's conjecture. In fact,
traditionally, that is, in Shannon's original and in some subsequent
papers, see \cite{Sha56, LL79}, it is
defined via a product that is often called the AND-product,
cf. e.g. \cite{AO}. Sometimes it is more convenient, however, to define graph
capacity in a complementary way, cf. e.g. \cite{CsK} (see Definition 11.3). The graph product involved then is the OR-product and the resulting notion is equivalent to the previous one defined for the complementary graph. To avoid confusion, we will call these two notions Shannon AND-capacity and Shannon OR-capacity, the latter being the one we will mostly use.

\begin{defi}\label{defi:prod}
Let $F$ and $G$ be two graphs. Both their AND-product $F\boxtimes G$ and OR-product $F\cdot G$ is defined on the Cartesian product $V(F)\times V(G)$ as vertex set.
The edge set of the OR-product is given by
$$E(F\cdot G)=\{\{(f,g),(f',g')\}: f,f'\in V(F), g,g'\in V(G), \{f,f'\}\in E(F)\ {\rm or}\  \{g,g'\}\in E(G)\}.$$
On the other hand, the edge set of the AND-product is given by
$$E(F\boxtimes G)=\{\{(f,g),(f',g')\}: f,f'\in V(F), g,g'\in V(G),$$$$ \{f,f'\}\in E(F)\ {\rm and}\ \{g,g'\}\in E(G),\ {\rm or}\ f=f', \{g,g'\}\in E(G),\ {\rm or}\ \{f,f'\}\in E(F), g=g' \}.$$
We denote the $t$-fold OR-product of a graph $G$ with itself by $G^t$, while the $t$-fold AND-product of $G$ with itself will be denoted by $G^{\boxtimes t}$.
\end{defi}

Denoting the complementary graph of a graph $H$ by $\overline{H}$, note that the above definitions imply that $\overline{F\cdot G}=\overline{F}\boxtimes\overline{G}$. In particular, $\omega(G^t)=\alpha(\overline{G^t})=\alpha(\overline{G}^{\boxtimes t})$, where $\omega(H)$ and $\alpha(H)$ denote the clique number and the independence number of graph $H$, respectively.

\begin{defi}\label{defi:Shcap}
The Shannon OR-capacity of a graph $G$ is defined as the always existing limit
$$C_{\rm OR}(G):=\lim_{t\to\infty}\sqrt[t]{\omega(G^t)}.$$
The Shannon AND-capacity is equal to
$C_{\rm AND}(G):=\lim_{t\to\infty}\sqrt[t]{\alpha(G^{\boxtimes t})}=C_{\rm OR}(\overline{G}).$
\end{defi}

We remark, that in information theory Shannon capacity is often defined as the
logarithm of the above values (to emphasize its operational meaning), but we
will omit logarithms as it is more customarily done in combinatorial
treatments. We also note, that all graphs in our discussions are meant to be
simple.


\begin{prop} \label{prop:hommon}
If $G$ and $H$ are two graphs such that $G\to H$, then $C_{\rm OR}(G)\le C_{\rm OR}(H).$
\end{prop}
\begin{proof}
Let $f:V(G)\to V(H)$ be a graph homomorphism and ${\mbf a}=a_1a_2\dots a_t,
{\mbf b}=b_1b_2\dots b_t$ be two adjacent vertices of $G^t$. Then for some $i$
we have $\{a_i,b_i\}\in E(G)$ implying $\{f(a_i),f(b_i)\}\in E(H)$ and thus
$\{f(a_1)f(a_2)\dots f(a_t),f(b_1)f(b_2)\dots f(b_t)\}\in E(H^t)$. Since our
graphs are simple, this implies $\omega(G^t)\le\omega(H^t)$ for every $t$ and thus the statement.
\end{proof}

\begin{cor}\label{cor:cator}
$$C_{\rm OR}(F\times G)\le \min\{C_{\rm OR}(F), C_{\rm OR}(G)\}.$$
\end{cor}

\begin{proof}
The claimed inequality follows from our discussion in the Introduction: since the appropriate projection maps define graph homomorphisms from $F\times G$ to $F$ and $G$, respectively, we have $F\times G\to F$ and $F\times G\to G$. By Proposition~\ref{prop:hommon} this implies the statement.
\end{proof}

Thus the following question is indeed valid: For what graphs $F$ and $G$ do we have equality in Corollary~\ref{cor:cator}? We elaborate on this question in the next two sections.

\section{On the possibilities of equality}\label{sect:canitbe}

If one is asked whether believing in equality in Corollary~\ref{cor:cator}
sounds plausible, then the most natural reaction seems to be to say ``no'' based
mainly on the fact that the answer to two somewhat similar questions is
negative, though neither is trivial. These two questions are the following.

Lov\'asz asked in his celebrated paper \cite{LL79}, whether Shannon
OR-capacity is multiplicative with respect to the OR-product. i.e.,
whether $$C_{\rm OR}(F\cdot G)=C_{\rm OR}(F)C_{\rm OR}(G)$$ holds for all
pairs of graphs $F$ and $G$. (Formally the
question was asked in the complementary language, but its mathematical content
was equivalent to this.) This was answered in the negative by Haemers in
\cite{Haemers}.

The second question is from Shannon's paper \cite{Sha56} and to present it in
our language we need the notion of {\em join} of two graphs.
\begin{defi}
The join $F\oplus G$ of graphs $F$ and $G$ has the disjoint union of $V(F)$ and
$V(G)$ as vertex set and its edge set is given by
$$E(F\oplus G)=E(F)\cup E(G)\cup\{\{a,b\}: a\in V(F), b\in V(G)\},$$
that is, $F\oplus G$ is the disjoint union of graphs $F$ and $G$ with all edges
added that has one endpoint in $V(F)$ and the other in $V(G)$.
\end{defi}
Shannon \cite{Sha56} proved that
$C_{\rm OR}(F\oplus G)\ge C_{\rm OR}(F)+C_{\rm OR}(G)$ for all pairs of
graphs $F$ and $G$ and formulated the conjecture that equality always
holds. This was refuted by Alon \cite{Alon} only four decades after the
question had been posed.

Two of the
graph parameters, the fractional chromatic number $\chi_f(G)$ and the Lov\'asz
theta number of the complementary graph (or strict vector chromatic number) $\bar{\vartheta}(G)=\vartheta(\bar{G})$ that
we mentioned in the Introduction as examples for graph parameters satisfying the Hedetniemi-type equality, i.e., for which we have
$$\chi_f(F\times G)=\min\{\chi_f(F),\chi_f(G)\}$$ and
$$\bar{\vartheta}(F\times G)=\min\{\bar{\vartheta}(F),\bar{\vartheta}(G)\},$$
respectively, also satisfy
$$\chi_f(F\cdot G)=\chi_f(F)\chi_f(G),\ \  \bar{\vartheta}(F\cdot
G)=\bar{\vartheta}(F)\bar{\vartheta}(G),$$
and
$$\chi_f(F\oplus G)=\chi_f(F)+\chi_f(G),\ \  \bar{\vartheta}(F\oplus
G)=\bar{\vartheta}(F)+\bar{\vartheta}(G).$$
(We remark that the chromatic number also trivially satisfies the second type
of these
equalities, i.e. $\chi(F\oplus G)=\chi(F)+\chi(G)$, but it does not satisfy the
first one. At the same time it does satisfy the inequality $\chi(F\cdot G)\le
\chi(F)\chi(G)$.)

\medskip
\par\noindent
The following notion (adapted again for our complementary language) is from
Zuiddam's recent paper \cite{Zuiddam} that borrows the terminology from
Strassen's work \cite{Str} which it is based on.

\begin{defi}\label{defi:asysp}
Let $S$ be a collection of graphs closed under the join and the OR-product
operations and containing the single vertex graph $K_1$. The {\em asymptotic
  spectrum} $Y(S)$ of $S$ is the set of all maps $\varphi: S\to R_{\ge 0}$ which satisfy for all $G,H\in S$ the following four properties:
\begin{itemize}
\item $\varphi(K_1)=1$

\item $\varphi(G\oplus H)=\varphi(G)+\varphi(H)$

\item $\varphi(G\cdot H)=\varphi(G)\cdot\varphi(H)$

\item if $G\to H$, then $\varphi(G)\le \varphi(H)$.

\end{itemize}

\end{defi}

Note that every $\varphi\in Y(S)$ provides an upper bound for the Shannon
OR-capacity of graphs in $S$. Indeed, the first two properties imply
$\varphi(K_n)=n$ for every $n$, which together with the fourth property imply
$\omega(G)\le\varphi(G)$ for every $G\in S$ and $\varphi \in Y(S)$. Using also
the third property we obtain
$$C_{\rm
  OR}(G)=\lim_{t\to\infty}\sqrt[t]{\omega(G^t)}\le\lim_{t\to\infty}\sqrt[t]{\varphi(G^t)}=\lim_{t\to\infty}\sqrt[t]{[\varphi(G)]^t}=\varphi(G).$$

Note also that $C_{\rm OR}(G)$ itself does not belong to $Y(S)$ by the above
mentioned results of Haemers \cite{Haemers} and Alon \cite{Alon}.

Building on Strassen's theory of asymptotic spectra, Zuiddam proved the
following surprising result (cf. also \cite{Fritz} for an independently found
weaker version).

\smallskip\par\noindent
\begin{thm}\label{thm:Zuid} (Zuiddam's theorem \cite{Zuiddam})
Let $S$ be a collection of graphs closed under the join and the OR-product
operations and containing the single vertex graph $K_1$. Let $Y(S)$ be the
asymptotic spectrum of $S$.
Then for all graphs $G\in S$ we have
$$C_{\rm OR}(G)=\min_{\varphi\in Y(S)}\varphi(G).$$
\end{thm}

\medskip\par\noindent
That is, Zuiddam's theorem states that the value of $C_{\rm OR}(G)$ is always
equal to the value of one of its ``nicely behaving'' upper bound
functions. Note that this would be trivial if $C_{\rm OR}(G)$ itself would be a
member of $Y(S)$, but we have already seen that this is not the case.
Zuiddam \cite{Zuiddam} gives a list of known elements of $Y(S)$. This list
(translated to our complementary language) includes the fractional chromatic
number, the Lov\'asz theta number of the complementary graph, the so-called
fractional Haemers bound (of the complementary graph) defined in
\cite{Blasiak} and further investigated in
\cite{BC, HTS}, and another parameter called fractional orthogonal rank
introduced in \cite{CMRSSW}. The fractional Haemers' bound also depends on a
field and as Zuiddam also remarks, a separation result
by Bukh and Cox \cite{BC} implies that this family of graph invariants has
infinitely many different elements.

\medskip\par\noindent
Now we are ready to prove our main result.

\medskip\par\noindent
\begin{thm}\label{thm:eitheror}
Either $$C_{\rm OR}(F\times G)=\min\{C_{\rm OR}(F),C_{\rm OR}(G)\}$$ holds for
graphs $F$ and $G$ or there exists some function $\varphi$ satisfying the
properties given in Definition~\ref{defi:asysp} for
which we have $$\varphi(F\times G)<\min\{\varphi(F),\varphi(G)\}.$$
\end{thm}

\medskip\par\noindent
In short, Theorem~\ref{thm:eitheror} states that either Shannon OR-capacity
satisfies the Hedetniemi-type equality, or if not,
then there is some much ``nicer behaving'' graph invariant, too, which also
violates it.

\smallskip\par\noindent
\proof
Consider two graphs $F$ and $G$
and let $S$ be
a class of graphs satisfying the conditions in Zuiddam's theorem and
containing all of $F$, $G$ and $F\times G$.
Then by Zuiddam's theorem there exists some $\varphi_0\in Y(S)$ for
which $$C_{\rm OR}(F\times G)=\varphi_0(F\times G).$$
Assume that $$\varphi_0(F\times G)=\min\{\varphi_0(F),\varphi_0(G)\}$$
holds. Then w.l.o.g. we may assume
$\varphi_0(F)\le\varphi_0(G)$ and thus $\varphi_0(F\times G)=\varphi_0(F).$

\smallskip\par\noindent
Since all elements in $Y(S)$ are upper bounds on Shannon OR-capacity, we also
have that
$$\min\{C_{\rm OR}(F),C_{\rm
  OR}(G)\}\le\min\{\varphi_0(F),\varphi_0(G)\}=\varphi_0(F).$$
But $$\varphi_0(F)=\varphi_0(F\times G)=C_{\rm OR}(F\times G),$$
so we have obtained
$$\min\{C_{\rm OR}(F),C_{\rm OR}(G)\}\le C_{\rm OR}(F\times G).$$
Since the opposite inequality is always true (by Corollary~\ref{cor:cator}) this
implies $$\min\{C_{\rm OR}(F),C_{\rm OR}(G)\}=C_{\rm OR}(F\times G).$$
Consequently, if the latter equality does not hold, then we must have
$\varphi_0(F\times G)\neq\min\{\varphi_0(F),\varphi_0(G)\}$ implying
$$\varphi_0(F\times G)<\min\{\varphi_0(F),\varphi_0(G)\}=\varphi(F_0)$$ by
$\varphi_0$ satisfying the fourth property in Definition~\ref{defi:asysp} and
the fact that $F\times G\to F$.
\hfill$\Box$

\medskip\par\noindent
{\bf Remark 1.} While the proof of Theorem~\ref{thm:eitheror}
  is rather simple
it may be worth noting how strong Zuiddam's theorem is on which it is
based. An illustration of this is given in the last fifteen minutes of
Zuiddam's lecture \cite{JZlect}, where he shows an equally simple proof of the
statement (translated to the language and notation we use here), that
$$C_{\rm OR}(F\cdot G)=C_{\rm OR}(F)C_{\rm OR}(G)\Leftrightarrow C_{\rm
  OR}(F\oplus G)=C_{\rm OR}(F)+C_{\rm OR}(G)$$ using his theorem.
In other words, in posession of Zuiddam's theorem Haemers' 1979 result
\cite{Haemers} about the non-multiplicativity of $C_{\rm OR}(G)$ with respect to
the OR-product already implies Alon's breakthrough result refuting Shannon's
conjecture that appeared only two decades later. $\Diamond$

\medskip\par\noindent
{\bf Remark 2.}
Graph parameters that satisfy the Hedetniemi-type equality, but violate the
conditions in Definition~\ref{defi:asysp} exist. A simple example is the
clique number that is not multiplicative with respect to the OR-product. (If
it was, then the Shannon-capacity problem would be trivial.)
A perhaps more artificial example is the reciprocal of the odd girth (taken to
be $0$ when the graph is bipartite) which
also satisfies the Hedetniemi-type equality but fails to do so with all but
the last one of the four conditions in Definition~\ref{defi:asysp}.
\hfill$\Diamond$

\section{A lower bound and identifying test cases}\label{sect:lbitc}

Due to the lack of knowledge of the Shannon capacity value for many graphs
(note that even that is not known whether the computational problem given by
it is decidable, see \cite{Alonsurv}), it is not entirely trivial how
to find a pair of graphs on which one could at least try checking whether
there is equality in Corollary~\ref{cor:cator} in any nontrivial way. In this
section we establish a general lower bound for $C_{\rm OR}(F\times G)$ and
present some graph pairs for which this lower bound is strictly smaller than the
upper bound $\min\{C_{\rm OR}(F),C_{\rm OR}(G)\}$. Whether either of the two
bounds is sharp in these cases remains an open problem.

\subsection{Lower bound}

\medskip\par\noindent
The following Proposition gives our lower bound.

\smallskip\par\noindent
\begin{prop}\label{prop:lb}
$$C_{\rm OR}(F\times G)\ge\max\{C_{\rm OR}(F'),C_{\rm OR}(G'): F'\subseteq F, F'\to G, G'\subseteq G, G'\to F\}.$$
\end{prop}

\begin{proof}
Let $F_0$ denote the subgraph $F'$ of $F$ that admits a homomorphism to $G$ with largest value of
$C_{\rm OR}(F')$. Let $G_0$ be the analogous
subgraph of $G$ obtained when we exchange the letters $F$ and $G$ in the
previous sentence. The statement is equivalent to the inequality
$C_{\rm OR}(F\times G)\ge\max\{C_{\rm OR}(F_0),C_{\rm OR}(G_0)\}$.

Thus it is enough to show $C_{\rm OR}(F\times G)\ge C_{\rm OR}(F_0)$, the same
argument will prove $C_{\rm OR}(F\times G)\ge C_{\rm OR}(G_0)$ when exchanging
the role of $F$ and $G$.

This readily follows from Proposition~\ref{prop:hommon}.  Indeed, since
$F_0\times G\subseteq F\times G$ we have
$C_{\rm OR}(F\times G)\ge C_{\rm OR}(F_0\times G)\ge C_{\rm OR}(F_0)$, where
the last inequality is a consequence of Proposition~\ref{prop:hommon} and the
fact that $F_0\to G$ implies $F_0\to F_0\times G$.
The latter follows by observing that if $f$ is a homomorphism from $F_0$ to $G$, then $\{u,v\}\in E(F_0)$ implies $\{(u,f(u)),(v,f(v))\}\in E(F_0\times G)$, therefore $f': u\mapsto (u,f(u))$ is a homomorphism from $F_0$ to $F_0\times G$.
This completes the proof.
\end{proof}

\begin{cor}
If $F\to G$ then $C_{\rm OR}(F\times G)=C_{\rm OR}(F)$.
\end{cor}

\begin{proof}
This is an immediate consequence of Proposition~\ref{prop:lb} and
Corollary~\ref{cor:cator}.
\end{proof}

\smallskip\par\noindent
For example, since a longer odd cycle always admits a
homomorphism to a shorter one (but not vice cersa) for arbitrary integers
$1\le k\le \ell$ we have
$C_{\rm OR}(C_{2k+1}\times C_{2\ell +1})=C_{\rm OR}(C_{2\ell +1})$. If $k=1, \ell=2$, then the above value is equal to
$C_{\rm OR}(C_5)=\sqrt{5}$ by the celebrated result of Lov\'asz
\cite{LL79} on the Shannon capacity of the $5$-cycle. We remark that the exact value of $C_{\rm OR}(C_{2\ell +1})$ is
unknown for all $\ell\ge 3$. A nontrivial result concerning these values is
proven by Bohman and Holzman \cite{BH} who showed that
$C_{\rm OR}(C_{2\ell +1})>2$ for every positive integer $\ell$.

\medskip
\par\noindent
Naturally, if we would like to ``test'' whether the inequality in
Corollary~\ref{cor:cator} can be strict then we need a pair of graphs $F$ and
$G$ for which the upper bound on $C_{\rm OR}(F\times G)$ provided by
Corollary~\ref{cor:cator} is strictly larger than the lower bound given in
Proposition~\ref{prop:lb}. As the exact value of Shannon capacity is known
only in a few nontrivial cases, finding such a pair is not a completely
trivial matter. We discuss this problem in the following subsection.

\subsection{Paley graphs and variants}

\medskip\par\noindent
\begin{defi}\label{defi:p17}
Let $q$ be an odd prime power satisfying $q\equiv 1 ({\rm mod}\ 4)$. The Paley graph
$P_q$ is defined on the elements of the finite field $F_q$ as vertices. Two
vertices form an edge if and only if their difference in $F_q$ is a square in
$F_q$.
\end{defi}

\medskip
\par\noindent
Note that the condition on $q$ ensures that $-1$ has a square root in $F_q$ and thus
$a-b$ is a square in $F_q$ if and only if $b-a$ is. Thus the definition is
indeed meaningful and results in a(n undirected) graph.
In the special case when $q$ itself is a prime number $p$, edges of $P_p$ are
between vertices whose difference is a quadratic residue modulo $p$.
\smallskip
\par\noindent
We also remark that $P_5$ is just the five-cycle $C_5$ and the graph $P_{17}$
is well-known to be the unique graph on $17$ vertices not having either a
clique or an independent set of size $4$, thus establishing the sharp lower
bound on the largest known diagonal Ramsey number $R(4,4)=18$ \cite{GG}. (In fact, $P_5\cong C_5$ is the unique graph establishing $R(3,3)\ge 6$. For more on the connection between Ramsey numbers and Shannon capacity, cf. \cite{AO, EMcT, NR}.)
\medskip
\par\noindent
Paley graphs are well-known to be self-complementary, vertex-transitive, and
edge-transitive, cf. e.g. \cite{Zhang, Peisert}. The first two of these
properties make them particularly useful for us by the following theorem of
Lov\'asz.

\begin{thm}\label{LLsc} (Lov\'asz \cite{LL79})
If $G$ is a vertex-transitive self-complementary graph on $n$ vertices, then
$$C_{\rm OR}(G)=\sqrt{n}.$$
\end{thm}

\medskip
\par\noindent
Thus we have $C_{\rm OR}(P_q)=\sqrt q$ for all prime
powers $q\equiv 1 ({\rm mod}\ 4)$.

\medskip
\par\noindent
Let $p\equiv 1 ({\rm mod}\ 4)$ be a prime number.
The value of the clique number (or equivalently, the independence number) of
$P_p$ is not known and determining it is a well-known unsolved problem in number theory, the
conjectured value being $O((\log  p)^2)$, cf. \cite{LLGG}. It is not hard to
see (and also follows from Theorem~\ref{LLsc} above) that
$\omega(P_p)\le \sqrt{p}$. Improving this bound by $1$ for infinitely
many primes $p$ was already a nontrivial task that was achieved by Bachoc,
Matolcsi, and Ruzsa \cite{BMR} only a few years ago. Recently Hanson and
Petridis \cite{HP} managed to improve this substantially by
proving the general upper bound
$$\omega(P_p)\le\frac{\sqrt{2p-1}+1}{2}<\frac{\sqrt{2p}+1}{2}.$$
Notice that this upper bound immediately implies
$$\chi(P_p)>\frac{2p}{\sqrt{2p}+1}$$
by $\chi(P_p)\ge\frac{|V(P_p)|}{\alpha(P_p)}=\frac{|V(P_p)|}{\omega(P_p)}.$
This in turn gives $\chi(P_p)>\sqrt{p}+1\ge\lceil{\sqrt{p}}\rceil,$ whenever
$p>20$ meaning
that for primes at least $20$ the largest subgraph of $P_p$ that can be
colored with $\lceil{\sqrt{p}}\rceil$ colors has strictly fewer vertices
than $P_p$ itself. There are only three primes of the form $4k+1$ below $20$: $5, 13, 17$.
As already mentioned above, $P_{17}$ is well-known from \cite{GG} to be the
graph establishing the largest known diagonal Ramsey number $R(4,4)=18$, that is, it
has no clique or independent set on more than $3$ vertices. In fact, this also follows
from the Hanson-Petridis bound as well as $\alpha(P_{13})\le 3$. Therefore we have
$\chi(P_{17})\ge\lceil\frac{17}{3}\rceil=6>\lceil\sqrt{17}\rceil=5$ as well as
$\chi(P_{13})\ge\lceil\frac{13}{3}\rceil=5>\lceil\sqrt{13}\rceil=4$.
(Obviously, the analogous inequality does not hold for $P_5\cong C_5$.)
\medskip
\par\noindent
This suggests that if we knew that deleting a vertex of a Paley graph $P_p$
its Shannon capacity becomes already smaller than $C_{\rm OR}(P_p)=\sqrt{p}$,
then we could conclude that for $m=\lceil{\sqrt{p}}\rceil$ the graph pair $(P_p,K_m)$ has the property, that the lower bound of Proposition~\ref{prop:lb} on
$C_{\rm OR}(P_p\times K_m)$ is strictly smaller than its upper bound from
Corollary~\ref{cor:cator}, which is $\sqrt{p}$ in this case. (Here we use that
$\max\{C_{OR}(H): H\subseteq P_p, H\to K_m\}=\max\{C_{OR}(H): H\subseteq P_p, \chi(H)\le m\}$, while $\max\{C_{OR}(T): T\subseteq K_m, T\to P_p\}\le m-1$ if
$\omega(P_p)<m$.)
Let us denote the graph we obtain from $P_p$ after deleting a vertex by $Q_{p-1}$. (Note that by the vertex-transitivity of $P_p$ it does not matter which vertex is deleted.)
Unfortunately, we do not have a proof that $C_{\rm OR}(Q_{p-1})<C_{\rm
  OR}(P_p)$ always holds. Though we believe it is true for any prime $p$ (of the form $4k+1$) it is
clear that this will not follow simply from the symmetry properties of $P_p$
(that one might believe at first sight), as the analogous statement is not
true for all prime powers $q$. Indeed, if $q=p^k$ for $k$ even, then it is
known that $\omega(P_q)=\sqrt{q}$ \cite{Blokhuis}, (cf. also \cite{BMR}) and
by vertex-transitivity this immediately implies $\omega(Q_{q-1})=\sqrt{q}$ as
well, that further implies $C_{\rm OR}(Q_{q-1})=\sqrt{q}$ (where analogously
to $Q_{p-1}$, $Q_{q-1}$ denotes the graph obtained by deleting a vertex of $P_q$).

\medskip
\par\noindent
Although we do not have a general proof for $C_{\rm OR}(Q_{p-1})<C_{\rm OR}(P_p)$,
using the computability of the Lov\'asz theta number we can decide in several
cases that this indeed happens. With the help of the online available Python code \cite{Stcode} to compute the Lov\'asz theta
number for specific graphs\footnote{I am grateful to Anna Gujgiczer for
  showing me how this code can be used and also for providing several of the
  required calculations.}, we can obtain for example the following values for
the first five relevant numbers:
$$\bar{\vartheta}(Q_{12})\approx 3.4927 < \sqrt{13},\ \
\bar{\vartheta}(Q_{16})\approx 4.0035 < \sqrt{17},$$
$$\bar{\vartheta}(Q_{28})\approx 5.3069 < \sqrt{29},\ \
\bar{\vartheta}(Q_{36})\approx 6.0025 < \sqrt{37},\ \
\bar{\vartheta}(Q_{40})\approx 6.3493 < \sqrt{41}.$$

\medskip
\par\noindent
It is worth calculating the Lov\'asz theta number also for graphs we obtain when deleting two vertices of a Paley graph. By the edge-transitivity of $P_p$ there are only two non-isomorphic such graphs (that are complementary to each other). We denote the one obtained when deleting two adjacent vertices by $Z_{p-2}^{(a)}$, and the one obtained when deleting two non-adjacent vertices by $Z_{p-2}^{(n)}$. Since $p$ is odd and $P_p$ is self-complementary, $Q_{p-1}$ is self-complementary as well (cf. \cite{Sachs}). Since any self-complementary graph $G$ on $n$ vertices has a clique of size $n$ in its second OR power (simply consider the vertices $(v,f(v))$, where $f$ is a complementing permutation), we have
$$C_{\rm OR}(Q_{p-1})\ge\sqrt{p-1}.$$
In fact, it is a natural question whether equality holds here. This is obviously so for $p=5$, but seems to be open for all other relevant values of $p$. (Note that the above numerical values of $\bar{\vartheta}(Q_{p-1})$ are all strictly larger than $\sqrt{p-1}$.)

\smallskip\par\noindent
By the foregoing, in case
$\max\{\bar{\vartheta}(Z_{p-2}^{(a)}),\bar{\vartheta}(Z_{p-2}^{(n)})\}<\sqrt{p-1}$,
it
implies $$\max\{C_{\rm OR}(Z_{p-2}^{(a)}),C_{\rm OR}(Z_{p-2}^{(n)})\}<\sqrt{p-1},$$ and thus  the graph pair $(Q_{p-1},K_{\ell})$
with $\ell=\lceil\sqrt{p-1}\rceil$ also has the property that our lower and
upper bounds do not coincide for $C_{\rm OR}(Q_{p-1}\times K_{\ell})$ provided
that $K_{\ell}\not\to Q_{p-1}$ (that we are ensured of by the Hanson-Petridis
upper bound on $\omega(P_p)$) and $Q_{p-1}\not\to K_{\ell}$, that is, $\chi(Q_{p-1})>\ell$. The latter condition is not yet true for $p=13$, but follows from the Hanson-Petridis bound \cite{HP} for all the larger relevant values of $p$. This is
particularly appealing when $\sqrt{p-1}$ is an integer itself. This happens in
several cases, starting (disregarding $p=2, 5$ that are not relevant for us) with $p=17, 37, 101$ (cf. sequence A002496 of The
Online Encyclopedia of Integer Sequences \cite{oeis}). Whether this sequence is
infinite (as it is believed to be) is a famous open problem in number theory
(one of the four problems called Landau's problems along with Goldbach's
conjecture, the twin-prime conjecture, and Legendre's conjecture).

\medskip
\par\noindent
Using again the Python code \cite{Stcode}, we obtain that
$$\max\{\bar{\vartheta}(Z_{15}^{(a)}),\bar{\vartheta}(Z_{15}^{(n)})\}=\min\{3.8726,
3.8849\}<4,$$
$$\max\{\bar{\vartheta}(Z_{35}^{(a)}),\bar{\vartheta}(Z_{35}^{(n)})\}=\min\{5.9128,
5.9251\}<6,$$
and
$$\max\{\bar{\vartheta}(Z_{99}^{(a)}),\bar{\vartheta}(Z_{99}^{(n)})\}=\min\{9.9496,
9.9574\}<10.$$

\par\noindent
Thus each of the graph pairs $(Q_{16},K_4), (Q_{36},K_6)$, and
$(Q_{100},K_{10})$ provide ``test cases'' for investigating the possibility of
equality in Corollary~\ref{cor:cator}. Let us stress again, that in the light
of Theorem~\ref{thm:eitheror} any proof showing for example
$C_{\rm OR}(Q_{16}\times K_4)<4$ would imply the existence of a graph
parameter that satisfies all the four conditions in
Definition~\ref{defi:asysp} and yet fails to satisfy the Hedetniemi-type
equality.

\section{Acknowledgement}
Many thanks are due to Anna Gujgiczer for calculating the values of
$\bar{\vartheta}$ for some relevant subgraphs of Paley graphs with the help of
Stahlke's code \cite{Stcode} and for explaining me how this code can be used.
I also thank Kati Friedl for a useful discussion.

\end{document}